\documentclass[11pt,leqno]{article}

\usepackage[utf8]{inputenc}
\usepackage{amsmath,amsfonts,amssymb,amssymb,amsthm,xcolor,bbm,xfrac, oldgerm,soul}
\usepackage[margin=2.5cm]{geometry}
\usepackage{hyperref}

\usepackage{fullpage}
\usepackage{graphicx}
\usepackage{mathrsfs}
\usepackage{xcolor}
\usepackage[colorinlistoftodos]{todonotes} 

\numberwithin{equation}{section}

\title{Convergence of linear solutions through convergence of periodic initial data}

\author{Harrison Gaebler\thanks{Department of Mathematics, University of North Texas, 1155 Union Circle 311430 Denton, Texas 76203-5017; \texttt{harrison.gaebler@unt.edu}}\quad Wesley R. Perkins\thanks{Division of Mathematics and Computer Science Lyon College, 2300 Highland Road, Batesville, Arkansas 72501; \texttt{wesley.perkins@lyon.edu}}}

\newcommand{\N}{\mathbb{N}}
\newcommand{\Z}{\mathbb{Z}}
\newcommand{\R}{\mathbb{R}}
\newcommand{\C}{\mathbb{C}}
\newcommand{\be}{\begin{equation}}
\newcommand{\ee}{\end{equation}}
\newcommand{\bes}{\begin{equation*}}
\newcommand{\ees}{\end{equation*}}
\newcommand{\ep}{\varepsilon}
\newcommand{\eps}{\varepsilon}
\renewcommand{\d}{\partial}
\renewcommand\a{\alpha}
\renewcommand\b{\beta}

\theoremstyle{definition}
\newtheorem{assumption}{Assumption}[section]
\newtheorem{definition}[assumption]{Definition}

\theoremstyle{plain}
\newtheorem{theorem}[assumption]{Theorem}
\newtheorem{lemma}[assumption]{Lemma}

\newtheorem{corollary}[assumption]{Corollary}

\newtheorem{remark}[assumption]{Remark}

\begin{document}

\maketitle

\begin{abstract}
When studying the stability of $T$-periodic solutions to partial differential equations, it is common to encounter subharmonic perturbations, i.e. perturbations which have a period that is an integer multiple (say $n$) of the background wave, and localized perturbations, i.e. perturbations that are integrable on the line.  Formally, we expect solutions subjected to subharmonic perturbations to converge to solutions subjected to localized perturbations as $n$ tends to infinity since larger $n$ values force the subharmonic perturbation to become more localized.  In this paper, we study the convergence of solutions to linear initial value problems when subjected to subharmonic and localized perturbations.  In particular, we prove the formal intuition outlined above; namely, we prove that if the subharmonic initial data converges to some localized initial datum, then the linear solutions converge.
\end{abstract}


\paragraph{Subject Class:} 

\section{Introduction}


In this paper, we will be studying initial value problems of the form
\[
\begin{cases}
    \d_t v = A v\\
    v(0)=g
\end{cases},
\]
where $A$ is a linear operator that generates a $C_0$ semigroup.  In particular, we will study how the convergence (in some meaningful sense described below) of periodic initial data informs the convergence of the solutions.  

This paper is motivated by several recent works regarding the stability of periodic waves; for example, see \cite{HJP21, JP21, JP22, HJPR24}.  When studying the stability of periodic waves for a general nonlinear partial differential equation, one must start by selecting an appropriate class of perturbations.  There are two important and widely studied classes of perturbations that arise naturally in applications.  If a periodic wave or coherent structure has period $T$, then one may consider subharmonic perturbations, i.e. $nT$-periodic perturbations for some integer $n$, or localized perturbations, i.e. perturbations which are integrable on the line.  One may also consider the special case of co-periodic perturbations by restricting subharmonic perturbations to the case when $n=1$. There has also been growing interest in studying general bounded perturbations, see \cite{R24}.

In \cite{HJPR24}, the authors establish a connection between the nonlinear solutions to the Lugiato-Lefever equation when a periodic wave is subjected to subharmonic perturbations and localized perturbations.  In particular, they show that they can formally recover the localized results established in \cite{HJPR23} by taking the limit as $n\to \infty$ of the subharmonic results established in \cite{HJPR24}.  (A similar connection is established for systems of reaction diffusion equations in \cite{JP21}.)   The goal of this paper is to rigorously verify the connection between localized results and the limit of subharmonic results, at least at the linear level for a very general linear operator.


\subsection{Definitions}
Following Corollary 1.6 in \cite{HJPR24}, Haragus et al. have discussion that indicates how they expect to ``formally'' recover their previously established localized result by taking the limit of the subharmonic result that they establish in Corollary 1.6.  Much of this discussion relies on the formal intuition that when $n$ increases, functions in $H^m_{\text{per}}(0,nT)$ look more like functions in $H^m(\R)$. We start by making the following definitions as a first step in making this intuition rigorous at the linear level. 

\begin{definition}
    For each function $g\in W^{k,p}_{\text{per}}(0,nT)$, we define the associated function $\tilde{g}\in W^{k,p}(\R)$ by $\tilde{g}(x)=g(x)$ for each $x\in\left[-\frac{nT}{2},\frac{nT}{2}\right]$ and $\tilde{g}(x)=0$ otherwise.
\end{definition}

One important observation that we will make use of throughout the paper is that if $g\in W^{k,p}_{\text{per}}(0,nT)$, then
\[
\|g\|_{W^{k,p}(0,nT)} = \|\tilde{g}\|_{W^{k,p}(\R)}.
\]

\begin{definition}{\label{norm_conv_over_period}}
Let $f_{1},f_{2},\ldots$ be a sequence of functions such that $f_{n}\in W^{k,p}_{\text{per}}(0,nT)$ for each $n\in\N$. We define this sequence to be ``norm convergent over a period'' or ``strongly convergent over a period'' to a function $f\in W^{k,p}(\R)$ if the sequence $\big(\tilde{f}_{n}\big)_{n}\subset W^{k,p}(\R)$ of associated functions converges in norm to $f$.
\end{definition}

This paper will focus on the spaces $L^1=W^{0,1}$, $L^2=W^{0,2}$, and $H^s=W^{s,2}$.  
We now provide an example to motivate the goal and structure of this paper. 

\subsection{Motivational Example}
The following example is rooted in work accomplished in \cite{HJPR23,HJPR24,ZUM22}.

Consider the Lugiatio-Lefever equation (LLE)
\begin{equation}\label{e:LLE}
\psi_t = -i\b \psi_{xx} - (1+i\a)\psi + i|\psi|^2\psi + F,
\end{equation}
where $\psi(x,t)$ is a complex-valued function depending on a temporal variable $t \geq 0$ and a spatial variable $x \in \R$, the parameters $\alpha,\beta$ are real, and $F$ is a positive constant. In 1987, the LLE was derived from Maxwell's equations in~\cite{LL87} as a model to study pattern formation within the optical field in a dissipative and nonlinear cavity filled with a Kerr medium and subjected to a continuous laser pump.  In this context, $\psi(x,t)$ represents the field envelope, $\alpha>0$ represents a detuning parameter, $|\beta|=1$ is a dispersion parameter, and $F>0$ represents the normalized pump strength. Note that the case $\beta=1$ 
is referred to as the ``normal'' dispersion case, while the case $\beta=-1$ 
is referred to as the ``anomalous'' dispersion case.  The LLE has recently become the subject of intense study in the recent physics literature, in part due to the fact that it has become a canonical model for high-frequency combs generated by microresonators in periodic optical waveguides; see, for example,~\cite{CGTM17} and references therein. 

The existence of spatially periodic stationary solutions of~\eqref{e:LLE}, as well as the nonlinear dynamics about them, has been studied by several recent works.
Such solutions $\psi(x,t)=\phi(x)$ correspond to $T$-periodic solutions of the profile equation
\begin{equation}\label{e:profile}
-i\b \phi'' - (1+i\a)\phi + i|\phi|^2\phi + F=0.
\end{equation}
Smooth periodic solutions of~\eqref{e:profile} have been constructed using perturbative arguments, as well as local and global bifurcation theory; for instance, see~\cite{DH18_2,DH18_1,Go17,HSS19,MR17,MOT1}.
{Most of the constructed periodic waves are unstable under general bounded perturbations~\cite{DH18_2}, though a class of periodic waves which are spectrally stable under general bounded perturbations has been identified in~\cite{DH18_1}. 
Nonlinear stability results have been obtained for localized perturbations in the recent works~\cite{HJPR23,ZUM22} and for co-periodic perturbations in~\cite{MOT2,SS19}. The results from~\cite{MOT2,SS19} can be extended to subharmonic perturbations 
provided spectral stability holds and the integer $n$ is fixed. In this $n$-dependent case, it turns out that both the allowable size of initial perturbations as well as the exponential rates of decay tend to zero as $n\to\infty$, leading to a lack of uniformity in the period of the perturbation. A linear stability result which holds uniformly in $n$ has been obtained in~\cite{HJP21}.  This result was upgraded to the nonlinear level in~\cite{HJPR24}.}  

The local dynamics about a given $T$-periodic stationary solution $\phi$ of~\eqref{e:LLE} can be captured by considering the perturbed solution 
\begin{equation}\label{e:pert}
\psi(x,t)=\phi(x)+{v}(x,t)
\end{equation}
of~\eqref{e:LLE}, where {$ v$ represents} some appropriately chosen perturbation.
Decomposing the solution $\phi=\phi_r+i\phi_i$ and the perturbation ${v}={v}_r+i{v}_i$ into their real and imaginary parts\footnote{In this example, we slightly abuse notation and write our complex functions $f$ in the form $f = \left(\begin{smallmatrix}f_r \\ f_i\end{smallmatrix}\right)$.}, we see that $\eqref{e:pert}$ is a solution of~\eqref{e:LLE} provided that the real-valued functions ${v}_r$ and ${v}_i$ satisfy the system
\begin{equation}\label{e:lin}
\partial_t\left(\begin{array}{c}{v}_r\\{v}_i\end{array}\right)=\mathcal{A}[\phi]\left(\begin{array}{c}{v}_r\\{v}_i\end{array}\right)+{\mathcal{N}}[\phi]({v}),
\end{equation}
where here $\mathcal A[\phi]$ is the (real) matrix differential operator
\begin{equation}\label{e:Aphi}
\mathcal A[\phi]=- I+\mathcal{J}\mathcal{L}[\phi],
\end{equation}
with
\[
\mathcal{J}=\left(\begin{array}{cc}0&-1\\1&0\end{array}\right),\quad
\mathcal{L}[\phi] = \left(\begin{array}{cc} -\b \d_x^2 - \a  + 3\phi_{r}^2 + \phi_{i}^2 & 2\phi_{r}\phi_{i} \\
  2\phi_{r}\phi_{i} & -\b \d_x^2 - \a  + \phi_{r}^2 + 3\phi_{i}^2\end{array}\right),
  \]
and where the nonlinearity is given by
\begin{align} \label{e:Nphi}
{\mathcal{N}}[\phi]\left({v}\right) = \mathcal{J}\left(\begin{array}{cc} 3{v}_{r}^2 + {v}_{i}^2 & 2{v}_{r}{v}_{i} \\
  2{v}_{r}{v}_{i} & {v}_{r}^2 + 3{v}_{i}^2\end{array}\right)\left(\begin{array}{c} \phi_r\\ \phi_i\end{array}\right) +
  \mathcal{J}\left|\left(\begin{array}{c}{v}_r\\{v}_i\end{array}\right)\right|^2\left(\begin{array}{c}{v}_r\\{v}_i\end{array}\right).
\end{align}

The following notion of spectral stability served as the main hypothesis for the uniform linear stability result for subharmonic perturbations in~\cite{HJP21} as well as for the nonlinear stability results for localized perturbations in~\cite{HJPR23} and for subharmonic perturbations in~\cite{HJPR24}.  This notion of spectral stability is a widely used notion, and specific periodic waves satisfying this definition are known to exist.  For more discussion on this definition, see \cite{HJP21,HJPR23,HJPR24}.

\begin{definition}\label{Def:spec_stab}
Let $T>0$.   A smooth $T$-periodic stationary solution $\phi$ of~\eqref{e:LLE} is said to be \emph{diffusively spectrally stable} provided the following conditions hold:
\begin{enumerate}
\item the spectrum of the linear operator $\mathcal{A}[\phi]$, given by~\eqref{e:Aphi} and acting in $L^2(\R)$, satisfies 
\[
\sigma_{L^2(\R)}(\mathcal{A}[\phi])\subset\{\lambda\in\C:\Re(\lambda)<0\}\cup\{0\};
\]
\item there exists $\theta>0$ such that for any $\xi\in[-\pi/T,\pi/T)$ the real part of the spectrum of the Bloch operator 
$\mathcal{A}_\xi[\phi]:=\mathcal{M}_\xi^{-1}\mathcal{A}[\phi]\mathcal{M}_\xi$, acting on $L^2_{\rm per}(0,T)$, satisfies
\[
\Re\left(\sigma_{L^2_{\rm per}(0,T)}(\mathcal{A}_\xi[\phi])\right)\leq-\theta \xi^2,
\]
where here $\mathcal{M}_\xi$ denotes the multiplication operator $\left(\mathcal{M}_\xi f\right)(x)=e^{i\xi x}f(x)$;
\item $\lambda=0$ is a simple eigenvalue of the Bloch operator $\mathcal{A}_0[\phi]$, and the derivative $\phi'\in L^2_{\rm per}(0,T)$ of the periodic wave is an associated eigenfunction.
\end{enumerate}
\end{definition} 

The following is a novel corollary of the main results proven in \cite{HJPR23,HJPR24}.  It serves to motivate the overarching goal of this paper; namely, it rigorously connects the convergence of the subharmonic solutions over a period to the localized solutions.

\begin{corollary}
Let $T>0$ and suppose $\phi(x)$ is a smooth $T$-periodic solution of \eqref{e:LLE} that is diffusively spectrally stable.
Let $\eps := \min\{\eps_1,\eps_2\}$, where $\eps_1$ is defined as in \cite{HJPR23}, Theorem 1.3, and $\eps_2$ is defined as in \cite{HJPR24}, Theorem 1.4.  Now, let $v_0\in L^1(\R)\cap H^4(\R)$  be such that $\|v_0\|_{L^1(\R)\cap H^4(\R)}<\frac{\eps}{2}$, and let $v_{n,0}\in H^2_{\rm per}(0,nT)$ be a sequence of functions that is norm convergent over a period to $v_0$, as defined in Definition \ref{norm_conv_over_period}.  Then, there exists a function $v(x,t)$ such that $u(x,t)=\phi(x)+v(x,t)$ is the unique global solution of \eqref{e:LLE} satisfying the initial condition $u(0)=\phi + v_0$ and 
\[
\|v(t)\|_{L^2(\R)}\leq C_1(1+t)^{-1/4}\|v_0\|_{L^1(\R)\cap H^4(\R)}.
\]
For $n$ sufficiently large, there similarly exists a sequence of functions $v_n(x,t)$ such that $u_n(x,t)=\phi(x)+v_n(x,t)$ is the unique global solution to \eqref{e:LLE} satisfying the initial condition $u_n(0)=\phi + v_{n,0}$ and
\[
\|v_n(t)\|_{H^2_{\rm per}(0,nT)} \leq C_2 \|v_{n,0}\|_{L^1_{\rm per}(0,nT)\cap H^2_{\rm per}(0,nT)}.
\]
In particular, the solutions $(u_n(t))_n$ strongly converge over a period to the solution $u(t)$.
\end{corollary}
\begin{proof}
Since $v_{n,0}$ converges over a period to $v_0$, there exists an $N\in\N$ such that $n\geq N$ implies
\[
\|\tilde{v}_{n,0}-v_0\|_{L^1(\R)\cap H^4(\R)} < \frac{\eps}{2}.
\]
Thus, $n\geq N$ implies
\[
\begin{aligned}
\|v_{n,0}\|_{L^1(0,nT)\cap H^4(0,nT)} &= \|\tilde{v}_{n,0}\|_{L^1(\R)\cap H^4(\R)}\\
&\leq \|\tilde{v}_{n,0}-v_0\|_{L^1(\R)\cap H^4(\R)} + \|v_0\|_{L^1(\R)\cap H^4(\R)} < \eps.
\end{aligned}
\]
By utilizing \cite{HJPR23}, Theorem 1.3, and $\eps_2$ is defined as in \cite{HJPR24}, Theorem 1.4, we therefore find that there exists a function $v(x,t)$ such that $u(x,t)=\phi(x)+v(x,t)$ is the unique global solution of \eqref{e:LLE} satisfying the initial condition $u(0)=\phi + v_0$ and 
\[
\|v(t)\|_{L^2(\R)}\leq C_1(1+t)^{-1/4}\|v_0\|_{L^1(\R)\cap H^4(\R)}.
\]
For $n\geq N$, there similarly exists a sequence of functions $v_n(x,t)$ such that $u_n(x,t)=\phi(x)+v_n(x,t)$ is the unique global solution to \eqref{e:LLE} satisfying the initial condition $u_n(0)=\phi + v_{n,0}$ and
\[
\|v_n(t)\|_{H^2_{\rm per}(0,nT)} \leq C_2 \|v_{n,0}\|_{L^1_{\rm per}(0,nT)\cap H^2_{\rm per}(0,nT)}.
\]  
In order to show that the solutions $(u_n(t))_n$ converge in norm over a period to the solution $u(t)$, it suffices to show that the perturbations $(v_n(t))_n$ converge in norm over a period to the perturbation $v(t)$.\footnote{We compare the perturbations rather than the solutions themselves because we do not know what space the solution $u$ belongs to.}  In particular, observe that, for each $t\geq 0$, 
\[
\begin{aligned}
    \|\tilde{v}_n(t)-v(t)\|_{L^1(\R)\cap H^4(\R)} &\leq \|\tilde{v}_n(t)\|_{L^1(\R)\cap H^4(\R)} + \|v(t)\|_{L^1(\R)\cap H^4(\R)}\\
    &\leq C_1\|v_{n,0}\|_{L^1(0,nT)\cap H^4(0,nT)} + C_2(1+t)^{-1/4}\|v_0\|_{L^1(\R)\cap H^4(\R)}\\
    &\leq C\eps,
\end{aligned}
\]
as desired.  
\end{proof}

At this point, we make a few remarks on the above corollary.  First, the above analysis was not unique and could trivially be carried out for similar equations where such localized and subharmonic results have been established; for example, systems of reaction diffusion equations \cite{JP21}.  Second, the above result heavily relied on the detailed and delicate analysis contained in \cite{HJPR23,HJPR24}.  In turn, this means that it heavily relied on the assumption that $\|v_0\|_{L^1(\R)\cap H^4(\R)}$ was sufficiently small enough to guarantee all of the necessary assumptions were met in order to use the previous results.  

It is natural to question whether the norm convergence over a period of the initial data is sufficient to show the nonlinear solutions strongly converge over a period when such previous results do not exist or when the initial data is not necessarily small.  In the case of the former, one would likely need to do a detailed study of the nonlinear asymptotic behavior of the solutions, which is highly nontrivial.  In the case of the latter, one would likely have to assume that the solutions exist (in the absence of previous results) and then use the structure of the nonlinear PDE itself to show convergence of the solutions.  As a partial step in that direction, we establish the following general result for the linear evolution of a convergent sequence of initial data.

\subsection{Main Result}

Let us consider the IVPs:
\be \label{IVP} 
\begin{cases} \partial_{t}v = A_{n}v \\ v(0)=g_{n}
\end{cases} \text{ and } \begin{cases} \partial_{t}v=Av \\ v(0)=g
\end{cases}.
\ee
We make the following assumptions regarding the linear operators $A_n$ and $A$ in \eqref{IVP}.

\begin{assumption}{\label{Assumption}}
We assume that:
\begin{itemize}
\item[(i)]  For each $n\in\N$, $A_{n}$ generates a $C_{0}$-semigroup on $L^{2}_{\rm per}(0,nT)$.  Additionally, for each $n\in\N$ and for each $\xi$ in the discrete set $\Omega_{n}=\{\xi\in[\frac{-\pi}{T},\frac{\pi}{T}) \mid e^{i \xi nT}=1\}$,  
$A_{n,\xi}:=\mathcal{M}_\xi^{-1}A_n\mathcal{M}_\xi$ generates a $C_{0}$-semigroup on $L^{2}_{\rm per}(0,T)$, where here $\mathcal{M}_\xi$ denotes the multiplication operator $\left(\mathcal{M}_\xi f\right)(x)=e^{i\xi x}f(x)$.

\item[(ii)]  $A$ generates $C_{0}$-semigroups on both $L^{2}(\R)$ and $L^{1}(\R)\cap H^{s}(\R)$ for $s>2$.  Additionally, for each $\xi\in\left[-\frac{\pi}{T},\frac{\pi}{T}\right)$, $A_{\xi}:=\mathcal{M}_\xi^{-1}A\mathcal{M}_\xi$ generates a $C_{0}$-semigroup on $L^{2}_{\rm per}(0,T)$.

\item[(iii)]  For each $n\in\N$, for each $\xi\in\Omega_{n}$, and for each $t\geq 0$, $e^{tA_{n,\xi}}=e^{tA_{\xi}}$.
\end{itemize}
\end{assumption}

\begin{remark}
    In practice, we consider the case where the $A_n=A$ are the same differential (or pseudo-differential) operators considered on the different Banach spaces $L^2_{\rm per}(0,nT)$ and $L^2(\R)$.  
    In this case, the above assumption can be simplified to the assumption that $A$ generates $C_0$-semigroups on $L^2(\R)$, $L^2_{\rm per}(0,nT)$ for each $n\in\N$, and (explicitly for Theorem \ref{main_result_linear}) $L^{1}(\R)\cap H^{s}(\R)$ for $s>2$, and that $A_\xi$ generates a $C_0$-semigroup on $L^2_{\rm per}(0,T)$ for each $\xi\in[-\pi/T,\pi/T)$. For examples, see \cite{HJP21,JP21,JP22, HJPR24}.
\end{remark}

With Assumption \ref{Assumption} in place, the IVPs \eqref{IVP} have the unique solutions
\begin{equation}\label{e:lin_solutions}
v_n(x,t) = e^{tA_n}g_n(x), \quad\quad v(x,t)=e^{tA}g(x),
\end{equation}
respectively. 
The main theorem of this section is then that, for each fixed $t\geq 0$, $(v_{n}(x,t))_{n=1}^{\infty}$ is norm convergent over a period to $v(x,t)$ if $(g_{n}(x))_{n=1}^{\infty}$ is norm convergent over a period to $g(x)$ and if the correct hypotheses are in place. It will be necessary for us to consider for each $n\in\N$ the auxiliary IVPs:
\be \label{e:IVP_mid}
\begin{cases} \partial_{t}v= Av \\ v(0)=\tilde{g}_{n}
\end{cases} 
\ee
which have the unique solutions $w_{n}(x,t)=e^{tA}\tilde{g}_{n}(x)$. We may now formulate our theorem as follows.

\begin{theorem}{\label{main_result_linear}}
Suppose that the following statements are true.
\begin{itemize}
\item[(i)] Assumption \ref{Assumption} holds.
\item[(ii)] For each $n\in\N$, we have that $g_{n}\in L^{1}_{\rm per}(0,nT)\cap H^s_{\rm per}(0,nT)$ for some $s>2$.
\item[(iii)] The sequence $(g_n)_n$ converges in $L^2(\R)$-norm over a period to $g\in L^{2}(\R)$.
\item[(iv)] For each $t\geq 0$, we have that $(\tilde{v}_{n}(x,t))_{n=1}^{\infty}$ is dominated by a function $h_{t}(x)\in L^{2}(\R)$.
\end{itemize}
Then, it follows that $\tilde{v}_{n}(x,t)\overset{L^{2}_{x}(\R)}{\longrightarrow} v(x,t)$ for each $t\geq 0$. Moreover, this convergence is uniform with respect to $t$ if the $C_{0}$-semigroup $\{e^{tA}\}_{t\geq 0}$ on $L^{2}(\R)$ is bounded.
\end{theorem}

\subsection{Outline of the Paper}

In Section 2, we first summarize properties of the Bloch transform.  We proceed to prove three important preparatory lemmas necessary to prove our main result.  In Section 3, we prove our main result, Theorem \ref{main_result_linear}.  Finally, in Section 4, we discuss the possibility of generalizing the result to weak convergence, rather than strong convergence.  We make use of the Banach-Saks theorem to prove a corollary regarding weak convergence.

\subsection{Notation}

For convenience, we introduce the following notations for each $n\in\N$ and $s>2$
\[
L^1_n := L^1_{\rm per}(0,nT),\quad\quad L^2_n := L^2_{\rm per}(0,nT),\quad\quad H^s_n := H^s_{\rm per}(0,nT).
\]
We equip $L^1_n$ and $L^2_n$ with the norms
\[
\|f\|_{L^1_n} := \int_{-\frac{nT}{2}}^{\frac{nT}{2}}|f(x)|\,dx, \quad\text{and}\quad \|f\|_{L^{2}_{n}}^2=\int_{-\frac{nT}{2}}^{\frac{nT}{2}}|f(x)|^{2}\,dx,
\]
respectively.  We equip $H^s_n$ with the norm 
\[
\|f\|_{H^s_n} := \|(1+(\cdot)^2)^{\sfrac{s}{2}}\hat{f}(\cdot)\|_{L^2_n},
\]
where $\hat{f}$ represents the Fourier transform of $f$ on the torus.

Note that $n=1$ corresponds to $T$-periodic functions, i.e. $L^2_1 = L^2_{\rm per}(0,T)$, etc.


\paragraph{Acknowledgments:} The authors are grateful to Mathew Johnson for several orienting discussions.  WRP is further grateful to the University of North Texas Millican Colloquium for funding trips that provided opportunity to collaborate with HG.

\section{Preliminaries}
\subsection{Floquet Bloch Theory}
Below, we provide a brief overview of properties of the Bloch transform.  For more details, see \cite{HJP21, HJPR24}.

For each $n\in\N$, note that $|\Omega_{n}|=n$ and that the distance between any two adjacent $\xi\in\Omega_{n}$ is given by $\Delta\xi =\frac{2\pi}{nT}$, where $\Omega_n$ is defined in Assumption \ref{Assumption}.  
Moreover, every function $f\in L^{2}_{n}$ admits a decomposition of the form
\be \label{Bloch_torus} f(x)=\frac{1}{2\pi}\sum_{\xi\in\Omega_{n}}e^{i\xi x}\mathcal{B}_{T}(f)(\xi,x)\Delta\xi \ee
where, for each fixed $\xi\in[\frac{-\pi}{T},\frac{\pi}{T})$, $\mathcal{B}_{T}(f)(\xi,x)=\sum_{l\in\Z}e^{ix\frac{2\pi l}{T}}\hat{f}\left(\xi+\frac{2\pi l}{T}\right)\in L^{2}_{1}$ defines the Bloch transform of $f\in L^{2}_{n}$ with Fourier transform $\hat{f}(s)=\int_{\frac{-nT}{2}}^{\frac{nT}{2}}e^{-ist}f(t)dt$. Similarly, if $f\in L^{2}(\R)$, then there is the decomposition
\be \label{Bloch_line} f(x)=\frac{1}{2\pi}\int_{\frac{-\pi}{T}}^{\frac{\pi}{T}}e^{i\xi x}\mathcal{B}(f)(\xi,x)d\xi \ee
where, for each fixed $\xi\in[\frac{-\pi}{T},\frac{\pi}{T})$, $\mathcal{B}(f)(\xi,x)=\sum_{l\in\Z}e^{ix\frac{2\pi l}{T}}\hat{f}(\xi+\frac{2\pi l}{T})\in L^{2}_{1}$ defines the Bloch transform of $f\in L^{2}(\R)$ with Fourier transform $\hat{f}(s)=\int_{-\infty}^{\infty}e^{-ist}f(t)dt$. Note that \eqref{Bloch_torus} can be interpreted as some sort of Riemann sum approximation to \eqref{Bloch_line}. 

The Bloch transform allows us to rewrite the solutions \eqref{e:lin_solutions} to \eqref{IVP} as
\begin{align} \label{semigroup_solution_to_IVP1} 
v_{n}(x,t)&=e^{tA_{n}}g_{n}(x) \nonumber \\ &= \frac{1}{2\pi}\sum_{\xi\in\Omega_{n}}e^{i\xi x}\mathcal{B}_{T}(e^{tA_{n}}g_{n})(\xi,x)\Delta\xi =\frac{1}{2\pi}\sum_{\xi\in\Omega_{n}}e^{i\xi x}e^{tA_{n,\xi}}\mathcal{B}_{T}(g_{n})(\xi,x)\Delta\xi 
\end{align}
and
\be \label{semigroup_solution_to_IVP2} 
v(x,t)=e^{tA}g(x) = \frac{1}{2\pi}\int_{-\frac{\pi}{T}}^{\frac{\pi}{T}}e^{i\xi x}\mathcal{B}(e^{tA}g)(\xi,x)d\xi=\frac{1}{2\pi}\int_{-\frac{\pi}{T}}^{\frac{\pi}{T}}e^{i\xi x}e^{tA_{\xi}}\mathcal{B}(g)(\xi,x)d\xi.
\ee
Similarly, we may rewrite the solution to \eqref{e:IVP_mid} as
\[
w_{n}(x,t)=e^{tA}\tilde{g}_{n}(x)=\frac{1}{2\pi}\int_{\frac{-\pi}{T}}^{\frac{\pi}{T}}e^{i\xi x}e^{tA_{\xi}}\mathcal{B}(\tilde{g}_{n})(\xi,x)d\xi.
\]
These representations will be fundamental to helping us understand the convergence of the linear solutions over a period.  As mentioned above, the fact that \eqref{semigroup_solution_to_IVP1} resembles a Riemann sum approximation of \eqref{semigroup_solution_to_IVP2} will be key to our analysis.

\subsection{Preparatory Lemmas}

In order to prove Theorem \ref{main_result_linear}, we will make use of the following three short lemmas.

\begin{lemma}{\label{Blochs_equal}}
Fix $n\in\N$. Then, $\mathcal{B}_{T}(v_{n})(\xi,x)=\mathcal{B}(w_{n})(\xi,x)$ for each $\xi\in\Omega_{n}$.
\end{lemma}
\begin{proof}
Note first that $\hat{g}_{n}(y)=\int_{\frac{-nT}{2}}^{\frac{nT}{2}}e^{-ixy}g_{n}(x)dx=\int_{-\infty}^{\infty}e^{-ixy}\tilde{g}_{n}(x)dx=\hat{\tilde{g}}_{n}(y)$ so that, for each $\xi\in[-\frac{\pi}{T},\frac{\pi}{T})$, there is
\bes \mathcal{B}_{T}(g_{n})(\xi,x)=\sum_{l\in\Z}e^{ix\frac{2\pi l}{T}}\hat{g}_{n}\left(\xi+\frac{2\pi l}{T}\right) =\sum_{l\in\Z}e^{ix\frac{2\pi l}{T}}\hat{\tilde{g}}_{n}\left(\xi+\frac{2\pi l}{T}\right)=\mathcal{B}(\tilde{g}_{n})(\xi,x).\ees
Then, let $\xi\in\Omega_{n}$ and observe from Assumption \ref{Assumption} that
\begin{align*} \mathcal{B}_{T}(v_{n})(\xi,x)=\mathcal{B}_{T}(e^{tA_{n}}g_{n})(\xi,x)&= e^{tA_{n,\xi}}\mathcal{B}_{T}(g_{n})(\xi,x)\\ &=e^{tA_{\xi}}\mathcal{B}(\tilde{g}_{n})(\xi,x)=\mathcal{B}(e^{tA}\tilde{g}_{n})(\xi,x)=\mathcal{B}(w_{n})(\xi,x)\end{align*}
as is required.
\end{proof}

We will also need to use the following convergence property of the Bloch transform.

\begin{lemma}{\label{pw_convergence}}
Suppose that $f_{n}\overset{L^{2}(\R)}{\longrightarrow} f$. Then, there exists a subsequence $(f_{n_{q_{r}}})_{r=1}^{\infty}$ such that, for each $x\in\R$, $\mathcal{B}(f_{n_{q_{r}}})(\xi,x)$ converges pointwise a.e. in $\xi$ to $\mathcal{B}(f)(\xi,x)$.
\end{lemma}
\begin{proof}
Note that because $f_{n}\overset{L^{2}(\R)}{\longrightarrow}f$, there exists a subsequence $f_{n_{q}}$ that converges pointwise a.e. in $x$ to $f$ and is itself dominated by a function $h_{1}\in L^{2}(\R)$. Moreover, it follows from Plancharel's theorem that
\bes 0=\lim_{q\to\infty}\|f_{n_{q}}-f\|_{L^2(\R)}=
\lim_{q\to\infty}\big\|\hat{f}_{n_{q}}-\hat{f}\big\|_{L^2(\R)} \ees
so, by passing to a further subsequence $f_{n_{q_{r}}}$, we have that $\hat{f}_{n_{q_{r}}}$ converges pointwise a.e. in $\xi$ to $\hat{f}$ and is itself dominated by a function $h_{2}\in L^{2}(\R)$. Thus, for each $x\in\R$,
\begin{align*} \mathcal{B}(f)(\xi,x)&=\sum_{l\in\Z}e^{ix\frac{2\pi l}{T}}\int_{-\infty}^{\infty}e^{-i\left(\xi +\frac{2\pi l}{T}\right) y}\lim_{r\to\infty}f_{n_{q_{r}}}(y)dy \\ &=\sum_{l\in\Z}e^{ix\frac{2\pi l}{T}}\lim_{r\to\infty}\int_{-\infty}^{\infty}e^{-i\left(\xi+\frac{2\pi l}{T}\right) y}f_{n_{q_{r}}}(y)dy \\ &= \sum_{l\in\Z}e^{ix\frac{2\pi l}{T}}\lim_{r\to\infty}\hat{f}_{n_{q_{r}}}\left(\xi+\frac{2\pi l}{T}\right) \\ &=\lim_{r\to\infty}\sum_{l\in\Z}e^{ix\frac{2\pi l}{T}}\hat{f}_{n_{q_{r}}}\left(\xi+\frac{2\pi l}{T}\right)= \lim_{r\to\infty}\mathcal{B}(f_{n_{q_{r}}})(\xi,x) \end{align*}
for a.e. $\xi\in\left[-\frac{\pi}{T},\frac{\pi}{T}\right)$ by two applications of the dominated convergence theorem.
\end{proof}

Finally, we will need to use the fact that the Bloch transform is continuous in $\xi$ when applied to functions of sufficiently high regularity.

\begin{lemma}{\label{B_continuity_in_xi}}
Suppose that $f\in L^{1}(\R)\cap H^{s}(\R)$ for $s>2$. Then, $\mathcal{B}(f)(\xi,x)$ is continuous in $\xi$.
\end{lemma}
\begin{proof}
Recall that $\mathcal{B}(f)(\xi,x)=\sum_{l\in\Z}e^{ix\frac{2\pi l}{T}}\hat{f}\left(\xi+\frac{2\pi l}{T}\right)$ and note that $\hat{f}\in C_{0}(\R)$ from the Riemann-Lebesgue lemma because $f\in L^{1}(\R)$. It now follows that each summand of the above series is continuous in $\xi$ and it is therefore enough for us to prove that this series converges uniformly in $\xi$.

Let $\xi\in[-\frac{\pi}{T},\frac{\pi}{T})$ be given and note that, for each $l\in\Z$, we have that
\bes (2l-1)\frac{\pi}{T}\leq\xi+\frac{2\pi l}{T} < (2l+1)\frac{\pi}{T} \ees
from which it follows that, for $l\geq 1$, there is
\bes 0<(2l-1)\frac{\pi}{T}\leq\xi+\frac{2\pi l}{T} \implies \frac{1}{\xi+\frac{2\pi l}{T}}\leq\frac{1}{(2l-1)\frac{\pi}{T}}\ees
and, for $l\leq -1$, there is
\bes \xi+\frac{2\pi l}{T}<(2l+1)\frac{\pi}{T}<0 \implies \frac{1}{\left\vert \xi+\frac{2\pi l}{T}\right\vert}<\frac{1}{\left\vert (2l+1)\frac{\pi}{T}\right\vert}. \ees
It therefore follows that, for $l\geq 1$ and $s>0$, the summands of the series that defines $\mathcal{B}(f)(\xi,x)$ can be estimated by
\begin{align*} \left\vert e^{ix\frac{2\pi l}{T}}\hat{f}\left(\xi+\frac{2\pi l}{T}\right) \right\vert&=\left\vert \frac{e^{ix\frac{2\pi l}{T}}}{\left(\xi+\frac{2\pi l}{T}\right)^{s}}\left(\xi+\frac{2\pi l}{T}\right)^{s}\hat{f}\left(\xi+\frac{2\pi l}{T}\right)\right\vert \\ 
&\leq\frac{1}{\left|(2l-1)\frac{\pi}{T}\right|^{s}}\left\vert \left(1+\left(\xi+\frac{2\pi l}{T}\right)^{2}\right)^{s/2}\hat{f}\left(\xi+\frac{2\pi l}{T}\right) \right\vert \\ 
&\leq\frac{1}{\left\vert(2l-1)\frac{\pi}{T}\right\vert^{s}}\left\Vert (1+(\cdot)^{2})^{s/2}\hat{f}(\cdot)\right\Vert_{L^{\infty}(\R)} \\ 
&\leq\frac{1}{\left\vert(2l-1)\frac{\pi}{T}\right\vert^{s}}\left\Vert (1+(\cdot)^{2})^{s/2}\hat{f}(\cdot)\right\Vert_{H^{1}(\R)} \\ 
&=\frac{1}{\left\vert(2l-1)\frac{\pi}{T}\right\vert^{s}}\|f\|_{H^{s+1}(\R)}.
\end{align*}
The second-to-last inequality follows by the Sobolev embedding $H^1(\R) \hookrightarrow L^\infty(\R).$
The last equality follows from Plancherel's Theorem for $H^1(\R)$ functions.  In particular, let $\mathcal{F}$ and $\mathcal{F}^{-1}$ be the Fourier and inverse Fourier transforms, respectively.  Then,
\begin{align*}
\left\Vert (1+(\cdot)^{2})^{s/2}\hat{f}(\cdot)\right\Vert_{H^{1}(\R)} &= \left\Vert \mathcal{F}^{-1}\left((1+(\cdot)^{2})^{s/2}\hat{f}(\cdot)\right)\right\Vert_{H^{1}(\R)}\\
&= \left\Vert (1+(\cdot)^2)^{1/2}\mathcal{F}\mathcal{F}^{-1}\left((1+(\cdot)^{2})^{s/2}\hat{f}(\cdot)\right)\right\Vert_{L^{2}(\R)}\\
&= \left\Vert (1+(\cdot)^{2})^{\sfrac{(s+1)}{2}}\hat{f}(\cdot)\right\Vert_{L^{2}(\R)} = \|f\|_{H^{s+1}(\R)}
\end{align*}
There is a similar estimate for $l\leq -1$. Then, we have that
\begin{align*} |\mathcal{B}(f)(\xi,x)| &=\left|\sum_{l\in\Z}e^{ix\frac{2\pi l}{T}}\hat{f}\left(\xi+\frac{2\pi l}{T}\right)\right| \\ 
&\leq \|\hat{f}\|_{L^{\infty}(\R)}+\|f\|_{H^{s+1}(\R)}\left(\sum_{l\geq 1}\frac{1}{\left\vert(2l-1)\frac{\pi}{T}\right\vert^{s}}+\sum_{l\leq -1}\frac{1}{\left\vert(2l+1)\frac{\pi}{T}\right\vert^{s}}\right).
\end{align*}
When $s>1$, the Weierstrass M-test then implies that $\mathcal{B}(f)(\xi,x)$ converges absolutely and uniformly on $[-\frac{\pi}{T},\frac{\pi}{T})$.  Consequently, $\mathcal{B}(f)(\xi,x)$ is continuous in $\xi$ so long as $f\in L^1(\R)\cap H^s(\R)$ for $s>2$, as desired.
\end{proof}

\section{Linear stability}


We are now ready to prove Theorem \ref{main_result_linear}.

\begin{proof}[Proof of Theorem \ref{main_result_linear}]
Fix $t\geq 0$. Then, let $\ep>0$ be given and choose $N=N(t)\in\N$ so that 
\bes \|\tilde{g}_{n}(x)-g(x)\|_{L^{2}(\R)}<\frac{\ep}{\|e^{tA}\|_{L^{2}(\R)\to L^{2}(\R)}} \ees
for each $n\geq N$. It is straightforward to see that the desired convergence $\tilde{v}_{n}(x,t)\overset{L^{2}_{x}(\R)}{\longrightarrow} v(x,t)$ will be uniform with respect to $t$ if the $C_{0}$-semigroup $\{e^{tA}\}_{t\geq 0}$ on $L^{2}(\R)$ is bounded and, in particular, we have for each $n\geq N$ that
\begin{align*} \|\tilde{v}_{n}(x,t)-v(x,t)\|_{L^{2}_{x}(\R)} &\leq \|\tilde{v}_{n}(x,t)-w_{n}(x,t)\|_{L^{2}_{x}(\R)}+\|w_{n}(x,t)-v(x,t)\|_{L^{2}_{x}(\R)} \\ &\leq  \|\tilde{v}_{n}(x,t)-w_{n}(x,t)\|_{L^{2}_{x}(\R)}+\|e^{tA}\|_{L^{2}(\R)\to L^{2}(\R)}\|\tilde{g}_{n}(x)-g(x)\|_{L^{2}(\R)} \\ &\leq  \|\tilde{v}_{n}(x,t)-w_{n}(x,t)\|_{L^{2}_{x}(\R)}+ \ep\end{align*}
so that it remains to estimate $\|\tilde{v}_{n}(x,t)-w_{n}(x,t)\|_{L^{2}_{x}(\R)}$. Indeed, note that $w_{n}(x,t)\overset{L_{x}^{2}(\R)}{\longrightarrow} v(x,t)$ so that, by passing to a subsequence if necessary as in Lemma \ref{pw_convergence}, it follows that $(w_{n}(x,t))_{n=1}^{\infty}$ is dominated by a function $h_{1,t}(x)\in L^{2}(\R)$. Similarly, $(\tilde{v}_{n}(x,t))_{n=1}^{\infty}$ is by hypothesis dominated by a function $h_{2,t}(x)\in L^{2}(\R)$. Thus, $(\tilde{v}_{n}(x,t)-w_{n}(x,t))_{n=1}^{\infty}$ is dominated by $h(x)=|h_{1,t}(x)|+|h_{2,t}(x)|\in L^{2}(\R)$. It therefore follows from the dominated convergence theorem that
\begin{align*} \lim_{n\to\infty}\|\tilde{v}_{n}(x,t)-w_{n}(x,t)\|^{2}_{L^{2}_{x}(\R)}&=\lim_{n\to\infty}\int_{\R}\left\vert\tilde{v}_{n}(x,t)-w_{n}(x,t)\right\vert^{2}dx \\ &=\int_{\R}\lim_{n\to\infty}\left\vert\tilde{v}_{n}(x,t)-w_{n}(x,t)\right\vert^{2}dx = 0 \end{align*}
if $\tilde{v}_{n}(x,t)-w_{n}(x,t)$ converges pointwise a.e. in $x$ to zero. Now, in order to establish this pointwise a.e. convergence in $x$, let $\eta>0$ be given and, for each $x\in\R$, choose a number $K_{1}=K_{1}(x)\in\N$ such that $x\in\left[-\frac{nT}{2},\frac{nT}{2}\right]$ for each $n\geq K_{1}$. Then, label the members of $\Omega_{n}$ as $\xi_{1}<\xi_{2}<\ldots<\xi_{n}$ for each $n\geq K_{1}$, and note in this case that
\begin{align*} 2\pi(\tilde{v}_{n}(x,t)-w_{n}(x,t))&=2\pi(v_{n}(x,t)-w_{n}(x,t)) \\ &=\sum_{j=1}^{n}e^{i\xi_{j}x}\mathcal{B}_{T}(v_{n})(\xi_{j},x)\Delta\xi-\int_{-\frac{\pi}{T}}^{\frac{\pi}{T}}e^{i\xi x}\mathcal{B}(w_{n})(\xi,x)d\xi \\ &=\sum_{j=1}^{n}e^{i\xi_{j}x}\mathcal{B}(w_{n})(\xi_{j},x)\Delta\xi-\sum_{j=1}^{n-1}\int_{\xi_{j}}^{\xi_{j}+\Delta\xi}e^{i\xi x}\mathcal{B}(w_{n})(\xi,x)d\xi \\ &\qquad\qquad\qquad -\int_{-\frac{\pi}{T}}^{\xi_{1}}e^{i\xi x}\mathcal{B}(w_{n})(\xi,x)d\xi - \int_{\xi_{n}}^{\frac{\pi}{T}}e^{i\xi x}\mathcal{B}(w_{n})(\xi,x)d\xi \end{align*} 
where we have used Lemma \ref{Blochs_equal}. The above quantity is then equal to
\begin{align*}&\sum_{j=1}^{n-1}\int_{\xi_{j}}^{\xi_{j}+\Delta\xi}\left(e^{i\xi_{j}x}\mathcal{B}(w_{n})(\xi_{j},x)-e^{i\xi x}\mathcal{B}(w_{n})(\xi,x)\right)d\xi \\ &\qquad\qquad+ e^{i\xi_{n}x}\mathcal{B}(w_{n})(\xi_{n},x)\Delta\xi  -\int_{-\frac{\pi}{T}}^{\xi_{1}}e^{i\xi x}\mathcal{B}(w_{n})(\xi,x)d\xi - \int_{\xi_{n}}^{\frac{\pi}{T}}e^{i\xi x}\mathcal{B}(w_{n})(\xi,x)d\xi\end{align*}
where we note, by passing to a further subsequence if needed, that $\mathcal{B}(w_{n})(\xi,x)$ converges pointwise a.e. in $\xi$ to $\mathcal{B}(v)(\xi,x)$ by Lemma \ref{pw_convergence}. It follows from this pointwise convergence and the fact that $\mathcal{B}(v)(\xi,x)\in L_{\xi}^{\infty}\left[-\frac{\pi}{T},\frac{\pi}{T}\right)$ (by the representation \ref{Bloch_line}) that there exists a number $M=M(x)\geq 0$ such that $|\mathcal{B}(w_{n})(\xi,x)|\leq M$ for a.e. $\xi\in\left[-\frac{\pi}{T},\frac{\pi}{T}\right)$ and for all but finitely-many $n\in\N$. Thus, we may choose a positive integer $K_{2}\in\N$ such that $\left(\Delta\xi+\left(\xi_{1}+\frac{\pi}{T}\right)+\left(\frac{\pi}{T}-\xi_{n}\right)\right)<\frac{\eta}{M}$ for each $n\geq K_{2}$ and such that
\begin{align*} &\left\vert  e^{i\xi_{n}x}\mathcal{B}(w_{n})(\xi_{n},x)\Delta\xi -\int_{-\frac{\pi}{T}}^{\xi_{1}}e^{i\xi x}\mathcal{B}(w_{n})(\xi,x)d\xi - \int_{\xi_{n}}^{\frac{\pi}{T}}e^{i\xi x}\mathcal{B}(w_{n})(\xi,x)d\xi \right\vert \\ &\qquad\leq M\left(\Delta\xi+\left(\xi_{1}+\frac{\pi}{T}\right)+\left(\frac{\pi}{T}-\xi_{n}\right)\right)<\eta.\end{align*}
Furthermore, it follows from Egorov's theorem that there exists a closed (and therefore compact) subset $F\subset\left[-\frac{\pi}{T},\frac{\pi}{T}\right)$ with $\mu\left(\left[-\frac{\pi}{T},\frac{\pi}{T}\right)\setminus F\right)<\frac{\eta}{M}$ (here, $\mu$ denotes the Lebesgue measure) such that $\mathcal{B}(w_{n})(\xi,x)$ converges uniformly in $\xi$ to $\mathcal{B}(v)(\xi,x)$ on $F$. Note also that $\mathcal{B}(v)(\xi,x)$ is continuous in $\xi$ on $F$ (and is therefore uniformly continuous in $\xi$ on $F$) because $\|w_{n}(x,t)\|_{L^{1}_{x}(\R)\cap H_{x}^{s}(\R)}\leq \|e^{tA}\|_{L^{1}(\R)\cap H^{s}(\R)\to L^{1}(\R)\cap H^{s}(\R)}\|\tilde{g}_{n}(x)\|_{L^{1}(\R)\cap H^{s}(\R)}$ so that $\mathcal{B}(w_{n})(\xi,x)$ is from Lemma \ref{B_continuity_in_xi} continuous in $\xi$ for each $n\in\N$.

Now, note that there exist finitely-many pairwise disjoint intervals $\{I_{r}\}_{r=1}^{d}$ such that $\left(\left[-\frac{\pi}{T},\frac{\pi}{T}\right)\setminus F\right)\subset\bigcup_{r=1}^{d}I_{r}$ and $\sum_{r=1}^{d}\mu(I_{r})<\frac{2\eta}{M}$ and, for each $n\in\N$, define the set
\bes Q_{n}=\{j\in\{1,\ldots,n-1\} \mid [\xi_{j},\xi_{j}+\Delta\xi]\cap I_{r}\neq\emptyset \text{ for some } r\in\{1,\ldots,d\}\}. \ees
It follows that, if $j\in Q_{n}$, then either $[\xi_{j},\xi_{j}+\Delta\xi]$ is a subset of the interior of $I_{r}$ for some $r\in\{1,\ldots,d\}$, or $[\xi_{j},\xi_{j}+\Delta\xi]$ contains at least one endpoint of $I_{r}$ for some $r\in\{1,\ldots,d\}$. We may therefore write $Q_{n}=Q_{n}^{1}\cup Q_{n}^{2}$ where
\bes Q_{n}^{1}=\{j\in Q_{n} \mid [\xi_{j},\xi_{j}+\Delta\xi] \text{ is a subset of the interior of } I_{r} \text{ for some } r\in\{1,\ldots,d\}\} \ees
and where $Q_{n}^{2}=Q_{n}\setminus Q_{n}^{1}$. In particular, $|Q_{n}^{2}|\leq 4d$ and we may choose a number $K_{3}\in\N$ so that $\frac{16Md\pi}{nT}<\eta$ and so that $\left\vert e^{i\xi _{j} x}-e^{i\xi x}\right\vert<\frac{T\eta}{2\pi M}$ (by the uniform continuity of $\xi\mapsto e^{i\xi x}$ on $\left[-\frac{\pi}{T},\frac{\pi}{T}\right]$) for each $n\geq K_{3}$. Then, we have that
\begin{align*} &\left\vert\sum_{j=1}^{n-1}\int_{\xi_{j}}^{\xi_{j}+\Delta\xi}\left(e^{i\xi_{j}x}\mathcal{B}(w_{n})(\xi_{j},x)-e^{i\xi x}\mathcal{B}(w_{n})(\xi,x)\right)d\xi\right\vert \\ &\leq \sum_{j=1}^{n-1}\int_{\xi_{j}}^{\xi_{j}+\Delta\xi}\left\vert\left( e^{i\xi_{j}x}-e^{i\xi x}\right)\mathcal{B}(w_{n})(\xi_{j},x)\right\vert d\xi + \sum_{j=1}^{n-1}\int_{\xi_{j}}^{\xi_{j}+\Delta\xi}\left\vert e^{i\xi x}\left(\mathcal{B}(w_{n})(\xi_{j},x)-\mathcal{B}(w_{n})(\xi,x)\right)\right\vert d\xi  \end{align*}
where the first term above admits the upper bound
\bes M\sum_{j=1}^{n-1}\int_{\xi_{j}}^{\xi_{j}+\Delta\xi}\left\vert e^{i\xi_{j}x}-e^{i\xi x}\right\vert d\xi\leq M\frac{2\pi }{T}\frac{T\eta}{2\pi M}=\eta \ees
for each $n\geq K_{3}$, and where the second term is equal to
\begin{align*} &\sum_{j\in Q_{n}}\int_{\xi_{j}}^{\xi_{j}+\Delta\xi}\left\vert\left(\mathcal{B}(w_{n})(\xi_{j},x)-\mathcal{B}(w_{n})(\xi,x)\right)\right\vert d\xi + \sum_{j\notin Q_{n}}\int_{\xi_{j}}^{\xi_{j}+\Delta\xi}\left\vert \left(\mathcal{B}(w_{n})(\xi_{j},x)-\mathcal{B}(w_{n})(\xi,x)\right)\right\vert d\xi. \end{align*}
In particular, note that
\begin{align*} &\sum_{j\in Q_{n}}\int_{\xi_{j}}^{\xi_{j}+\Delta\xi}\left\vert\left(\mathcal{B}(w_{n})(\xi_{j},x)-\mathcal{B}(w_{n})(\xi,x)\right)\right\vert d\xi \\ &=\sum_{j\in Q^{1}_{n}}\int_{\xi_{j}}^{\xi_{j}+\Delta\xi}\left\vert\left(\mathcal{B}(w_{n})(\xi_{j},x)-\mathcal{B}(w_{n})(\xi,x)\right)\right\vert d\xi + \sum_{j\in Q^{2}_{n}}\int_{\xi_{j}}^{\xi_{j}+\Delta\xi}\left\vert\left(\mathcal{B}(w_{n})(\xi_{j},x)-\mathcal{B}(w_{n})(\xi,x)\right)\right\vert d\xi \\ &\leq 2M\sum_{r=1}^{d}\mu(I_{r})+ 8Md\Delta\xi < 2M\frac{2\eta}{M}+8Md\frac{2\pi}{nT} = 4\eta + \frac{16Md\pi}{nT}<5\eta\end{align*}
for each $n\geq K_{3}$. Finally, for $j\notin Q_{n}$, it follows that $\xi_{j}$ and $\xi$ in the integrand belong to $F$ because $[\xi_{j},\xi_{j}+\Delta\xi]\cap I_{r}=\emptyset$ for each $r\in\{1,\ldots,d\}$, and we may choose a number $K_{4}\in\N$ so that
\begin{align*} &\sum_{j\notin Q_{n}}\int_{\xi_{j}}^{\xi_{j}+\Delta\xi}\left\vert\mathcal{B}(w_{n})(\xi_{j},x)-\mathcal{B}(w_{n})(\xi,x)\right\vert d\xi \\ &\qquad\leq\eta+ \sum_{j\notin Q_{n}}\int_{\xi_{j}}^{\xi_{j}+\Delta\xi}\left\vert\mathcal{B}(v)(\xi_{j},x)-\mathcal{B}(v)(\xi,x)\right\vert d\xi \leq \eta+ \frac{2\pi}{T}\frac{\eta T}{2\pi}=2\eta \end{align*}
for each $n\geq K_{4}$ by first the uniform convergence of $\mathcal{B}(w_{n})(\xi,x)$ to $\mathcal{B}(v)(\xi,x)$ on $F$ and, second, the uniform continuity in $\xi$ on $F$ of $\mathcal{B}(v)(\xi,x)$. Thus, we have established that $2\pi|\tilde{v}_{n}(x,t)-w_{n}(x,t)|\leq 8\eta$ for each $n\geq\max\{K_{1},K_{2},K_{3},K_{4}\}$. In particular, it follows that every subsequence of the original sequence $(\tilde{v}_{n}(x,t)-w_{n}(x,t))_{n=1}^{\infty}$ of differences has a further subsequence that converges to zero. This implies that $(\tilde{v}_{n}(x,t)-w_{n}(x,t))_{n=1}^{\infty}$ itself converges to zero and, because $x\in\R$ was chosen arbitrarily, it now follows that $(\tilde{v}_{n}(x,t)-w_{n}(x,t))_{n=1}^{\infty}$ converges pointwise a.e. in $x$ to zero, as required.
\end{proof}

\section{Brief discussion}

It is natural to wonder if the hypotheses of Theorem \ref{main_result_linear} can be weakened in some suitable way. For example, if $\tilde{g}_{n}$ converges weakly in $L^{2}(\R)$ to $g$, then do the solutions $\tilde{v}_{n}(x,t)$ converge weakly in $L^{2}(\R)$ to $v(x,t)$? Indeed, for a fixed $t\geq 0$ and $\varphi\in L^{2}(\R)$, we do have that
\bes \lim_{n\to\infty}\langle \varphi, w_{n}-v\rangle =\lim_{n\to\infty}\langle (e^{tA})^{*}\varphi, \tilde{g}_{n}-g\rangle =0 \ees
and, thus, it only remains to estimate the term $\langle \varphi, \tilde{v}_{n}-w_{n}\rangle$ for large $n$. Unfortunately, there is no reason to expect that, for each $\varphi\in L^{2}(\R)$,
\bes \langle \varphi, \tilde{v}_{n}-w_{n}\rangle = \int_{\R}\varphi(x)(\tilde{v}_{n}(x,t)-w_{n}(x,t))dx \ees
is small for large $n$ if $\tilde{v}_{n}(x,t)-w_{n}(x,t)$ does not converge pointwise a.e. in $x$ to zero. In particular, we cannot hope for such pointwise a.e. in $x$ convergence because this requires norm convergence of the initial data, as is demonstrated by the proof of Theorem \ref{main_result_linear}. It is therefore unlikely (barring the development of a different argument altogether) that the weak convergence over a period of initial data implies the weak convergence over a period of the corresponding solutions. However, we can appeal to Theorem \ref{main_result_linear} for a positive result concerning certain averages of the initial data. Note that if $\tilde{g}_{n}$ converges weakly in $L^{2}(\R)$ to $g$, then it follows from the Banach-Saks theorem that there exists a subsequence $(g_{n_{j}})_{j=1}^{\infty}$ of the initial data such that the averages $G_{m}=\frac{1}{m}\sum_{j=1}^{m}\tilde{g}_{n_{j}}$ converge strongly in $L^{2}(\R)$ to $g$. In particular, $G_{m}$ is supported on $\left[-\frac{n_{m}T}{2},\frac{n_{m}T}{2}\right]$ for each $m\in\N$ and we therefore define for each $m\in\N$ the new function $G_{m}^{\text{per}}\in L^{2}_{n_{m}}$ to be the $n_{m}$-periodic function that coincides with $G_{m}$ on its support. It follows that $\widetilde{G}_{m}^{\text{per}}=G_{m}$ and, thus, we have the following corollary of Theorem \ref{main_result_linear}.

\begin{corollary}
    Suppose that the following statements are true for the functions that are mentioned in the preceding paragraph.
    \begin{itemize}
        \item[(i)] Assumption \ref{Assumption} holds.
        \item[(ii)] The sequence $(g_n)_n$ converges weakly in $L^2(\R)$ over a period to $g\in L^2(\R)$.
        \item[(iii)] For each $m\in\N$, we have that $G^{\text{per}}_{m}\in L_{n_{m}}^{1}(\R)\cap H_{n_{m}}^{s}(\R)$ for some $s>2$.
        \item[(iv)] For each $t\geq 0$, we have that $(\widetilde{V}_{m}(x,t))_{m=1}^{\infty}$ is dominated by a function $H_{t}(x)\in L^{2}(\R)$, where $V_{m}=e^{tA_{n_{m}}}G^{\text{per}}_{n_{m}}$.
    \end{itemize}
    Then, it follows that $\widetilde{V}_{m}(x,t)\overset{L^{2}_{x}(\R)}{\longrightarrow} v(x,t)$ for each $t\geq 0$. Moreover, this convergence is uniform with respect to $t$ if the $C_{0}$-semigroup $\}e^{tA}\}_{t\geq 0}$ on $L^{2}(\R)$ is bounded.
\end{corollary}
\begin{proof}
Apply Theorem \ref{main_result_linear} to the sequence $(\widetilde{G}^{\text{per}}_{m})_{m=1}^{\infty}=(G_{m})_{m=1}^{\infty}$ that converges strongly in $L^{2}(\R)$ to $g$.   
\end{proof}


\bibliographystyle{abbrv}
\bibliography{bibliography}

\end{document}